\documentclass{article}
\usepackage{amsmath}
\usepackage{amssymb}
\usepackage{indentfirst}
\usepackage{amsthm}
\usepackage{fancyhdr}
\usepackage{titlesec}
\usepackage{graphicx}
\usepackage{subfigure}
\title{Relative Morse Categorification Theory}
\author{Danning Lu\thanks{danninglu1994@pku.edu.cn, School of Mathematical Sciences, Peking University, Beijing, China} \and Xiaohan Yan\thanks{1300010605@pku.edu.cn, School of Mathematical Sciences, Peking University, Beijing, China}}
\date{}

\begin{document}
\newtheorem{defi}{Definition}
\newtheorem{thrm}{Theorem}
\newtheorem{lem}{Lemma}
\newtheorem{prop}{Proposition}
\newtheorem{rmk}{Remark}
\maketitle
\begin{abstract}
In this paper, we define a relative Morse complex for manifold with boundary using the handlebody decomposition of the manifold. We prove that the homology of the relative Morse complex is isomorphic to the relative singular homology. Furthermore, we construct $A_\infty$-category structure on the relative Morse complex by counting the trajectory trees among the critical points of different Morse functions. Our result generalizes Fukaya's construction on closed manifold and Manabu's construction of absolute homology on manifold with boundary.
\end{abstract}

\section{Introduction}

\subsection{Overview}
\par In 1980s, E. Witten \cite{Wit} first described the Morse complex with respect to Morse function $f$ and Riemannian metric $g$. The chain groups $CM_k(f)$ are free abelian groups generated by the critical points of $f$ with Morse index $k$, while the boundary homomorphisms are defined by counting the number of gradient trajectories flowing from one critical point $p$ to another one with a lower Morse index. It was proved that the Morse homology is isomorphic to the singular homology of the manifold. Then in 1990s, when considering the Floer homologies, K. Fukaya found that similar ideas could be used in Morse homology theory. In \cite{Fuk}, he defined the $A_\infty$-structure for the Morse Category, in which the objects were Morse functions and the morphisms $\text{Hom}(f_1, f_2)$ between two objects $f_1, f_2$ were the Morse Complex with respect to $f_1-f_2$. For generic Morse functions $f_0, \ldots, f_k$, let $p_i(i=0, \ldots, k-1)$ be a critical point of $f_i-f_{i+1}$, $p_k$ be a critical point of $f_0-f_k$. Denote by $\sharp\mathcal{M}(p_0, \ldots, p_{k-1}, p_k)$ the cardinality of the dimension-zero moduli space $\mathcal{M}(p_0, \ldots, p_k)$, i.e. the set of gradient trajectory trees starting from $p_0, \ldots, p_{k-1}$ and ending at $p_{k}$. Then the $A_\infty$-maps were given by
\[
m_k([p_{k-1}]\otimes\ldots\otimes[p_0]) = \sum_{p_k} \sharp\mathcal{M}(p_0, p_1, \ldots, p_k)[p_k],
\]
where the sum is over all critical points $p_k$ whose Morse index satisfies
\[
\mu(p_k)=\mu(p_0) + \cdots + \mu(p_{k-1})-(k-1)n + k - 2.
\]
The $A_\infty$-relationship can be attributed to the conclusion that the number of points with sign is zero in dimension-one moduli spaces' boundary. In fact, the cup product in Morse Category can be paralleled with the product structure on deRham complex in a way. In \cite{Con}, K. W. Chan, N. C. Leung and Z. N. Ma proved that the cup products on deRham complex have semiclassical limits as the cup product defined by counting gradient trajectory trees with respect to Morse functions.
\par Furthermore, the theory can be generalized to absolute homology for manifolds with boundary and relative homology. M. Schwarz \cite{MS} introduced a method of constructing the relative Morse complex. For manifold $M$, its submanifold $A$, and a Morse function $f$ on $M$, he proved that the Morse complex of $f$ on $A$ could be treated as a subcomplex of the Morse complex of $f$ on $M$ after doing some slight modifications. In this case, he was able to define relative Morse complex by using the quotient complex. Due to the way of defining the complex, it is easy to see that the homology of the relative Morse complex is exactly the relative homology of $(M, A)$. Besides, M. Akaho(see \cite{Akh2}) defined a Morse complex for manifolds with boundary, by examining certain critical points of cone end Morse functions and Riemannian metric. To be specific, let $N_i\times(0,1)$ be a collar neighborhood of the connected component $N_i$ of the manifold's boundary, and $r$ be the factor of $(0,1)$. Then a cone end Morse function $f$ is one that satisfies $f|_{N_i\times(0,1)}=r^2f_{N_i}+c_i$ in the collar neighborhood, while a cone end Riemannian metric is one that satisfies $g|_{N_i\times(0,1)}=r^2g_{N_i}+dr\otimes dr$. Then he used the internal critical points as well as the specific critical points on boundary $\gamma\in N_i$ such that $f|_{N_i}(\gamma)>0$ to generate the Morse complex. He proved that the Morse homology is isomorphic to the absolute singular homology of the manifold with boundary, through observing the handlebody decomposition of the manifold with respect to unstable manifolds of the critical points as well as the relative cycles. Moreover, he defined the cup product structure in \cite{AKh3}, serving as a prototype of $A_\infty$-categorification of absolute Morse homology theory for manifolds with boundary. Also, for compact manifolds with boundary, F. Laudenbach \cite{LB} introduced a method that allows one to define Morse complex whose homology is isomorphic to absolute or relative singular homology without the requirement in geometric settings.
\par The profound works given by the predecessors motivate our work in this paper. The aim of this paper is to consider the $A_\infty$ categorification of the relative Morse homology theory. However, the definition of the relative Morse homology in this paper is different from that of M. Schwarz. To be specific, we first transform the situation of a manifold and its submanifold to the situation of a manifold with boundary and its boundary. For manifold $\widetilde{M}$ with dimension $n$ and its submanifold $\widetilde{N}$, if the dimension of $\widetilde{N}$ is less than $n$, we can replace $\widetilde{N}$ by its dimension-$n$ tubular neighborhood $N$ while not changing its homotopy type. Let $M = \widetilde{M} - N$. Then by excision lemma, the homology of $(\widetilde{M}, \widetilde{N})$ is isomorphic to that of $(M\cup \partial M \times [0,\epsilon], \partial M \times [0,\epsilon])$, and thus isomorphic to $(M, \partial M)$.  Then for manifold $M$ with boundary, we choose special Riemannian metric $g$ satisfying $g|_{(0,1) \times\partial M} = g_{\partial M}+\frac{1}{r}dr\otimes dr$ near the boundary, and horn-end transverse Morse functions satisfying $f|_{[0,1)\times N_i}=f|_{N_i}-k_i r$, where $k_i$ are nonzero constants. With the above preparations, we define the relative Morse complex directly while avoiding usage of quotient complex as in \cite{MS}. The complex has the following form:
\[
CM_k(f)=\bigoplus_{q\in Cr_k^\bullet(f)}\mathbb{Z}[q],
\]
\[
\partial [q]=\sum_{q'\in Cr_{k-1}^\bullet(f)}\sharp\mathcal{M}(q,q') [q'], \text{for }q\in Cr_k^\bullet(f),
\]
where $\mathcal{M}(q,q')$ is the moduli space of gradient trajectories starting at $q$ with Morse index $k$ and ending at $q^\prime$ with Morse index $k-1$. The $Cr_k^\bullet(f)$ here only embraces those critical points in $M^\circ$ and on `positive boundaries', but not including all the ones with Morse index $k$. After that, we prove the homology of this complex is isomorphic to the manifold's singular homology relative to its boundary. In this step, we use the handlebody decomposition of the manifold and treat the homology of this decomposition as a bridge linking the relative Morse homology and the relative singular homology.
\begin{thrm} \label{t1}
The relative Morse homology $H_*(CM_*(f),\partial_*)$ is isomorphic to the relative singular homology $H_*(M,\partial M)$.
\end{thrm}
\par At last, we define the $A_\infty$-structure, still using the technique of counting the gradient trajectories. We first discuss the cup product structure and prove the Leibnitz rule as an example of the more complicated $A_\infty$-relationships. Then, let $F=(f_0, f_1, \ldots, f_k)$ be a generic sequence of functions on $M$ such that $f_i-f_j$ is transverse Morse function for $i\neq j$. Let $q_k$ be a critical point in $Cr^\bullet(f_0-f_k)$ and $q_i$ be a critical point in $Cr^\bullet(f_i-f_{i+1})$ for $i=0, 1, \ldots, k-1$. When the Morse indices satisfy $\mu(q_k) = \mu (q_0) + \cdots + \mu (q_{k-1}) - (k - 1)n + k - 2$, the moduli space $\mathcal{M}(q_0, q_1, \ldots, q_k)$ of gradient trajectory trees originating from $q_0, \ldots, q_{k-1}$ and ending at $q_k$ is a dimension-zero manifold. We define the $A_\infty$-map as
\[
m_k: CM_*(f_{k-1}- f_k)\otimes \cdots \otimes CM_*(f_0- f_1) \rightarrow CM_*(f_0- f_k)[2-k]
\]
\[
m_k([q_{k-1}]\otimes\ldots\otimes[q_0]) = \sum_{q_k} \sharp\mathcal{M}(q_0, q_1, \ldots, q_k)[q_k],
\]
where $\sharp\mathcal{M}(q_0, q_1, \ldots, q_k)$ means the number of points with sign in $\mathcal{M}(q_0, q_1, \ldots, q_k)$. Actually we can see here that the degree of a critical point $q_i$ should be $\mu(q_i) - n$ so that the $A_\infty$-relationships hold true. And we prove the complete $A_\infty$-relationships by considering the boundaries of the moduli spaces $\mathcal{M}(q_0, q_1, \ldots, q_k^\prime)$ for all $q_k^\prime$ such that $\mu(q_k^\prime) = \mu (q_0) + \cdots + \mu (q_{k-1}) - (k - 1)n + k - 3$. Finally, for $\mathfrak{MC}(M)$, the category of transverse Morse functions on $M$, we will have
\begin{thrm} \label{t2}
$\mathfrak{MC}(M)$ is endowed with the $A_\infty$-category structure using $m_k$.
\end{thrm}

\subsection{Structure of the Paper}
In section 2, we introduce the definitions about horn-end transverse Morse functions, which are the functions we use throughout the paper, and its critical points. After that, we construct a decomposition(handlebody decomposition) of $M$. Then in section 3, we construct the relative Morse complex of $M$, and prove that the homology of the complex we defined is isomorphic to the relative singular homology using handlebody decomposition. And lastly in section 4, we categorify the relative Morse theory by defining the Morse Category, and endow the category with $A_\infty$-structure.

\section{Horn-end Transverse Morse Function and Handlebody Decomposition}
\par In this section, we introduce the definition of horn-end transverse Morse function $f$ and its critical points, and then introduce the handlebody decomposition of $M$ with respect to $f$.

\subsection{Horn-end Transverse Morse Function and Horn-end Riemannian Metric}
\par Let $M$ be a compact, oriented manifold of dimension $n$ with boundary $\partial M$. Denote by $N_i(i\in \Lambda)$ the connected components of $\partial M$. We fix a collar neighborhood $[0,1)\times N_i\subset M$, and denote by $r$ the standard coordinate on the $[0,1)$ factor.
\begin{defi}
A smooth function $f$ on $M$ is called a Morse function if it satisfies the following conditions:\\
(1)$f| _{M\setminus\partial M}$ is a Morse function on the manifold $M\setminus\partial M $;\\
(2)$f| _{\partial M}$ is a Morse function on the manifold $\partial M$.\\
And it is called a transverse Morse function if it satisfies the extra condition:\\
(3)For any point $x$ on the collar neighborhood, $\left.-\frac{\partial f}{\partial r}\right|_{x}\neq 0$.
\end{defi}

\begin{defi}
For a transverse Morse function $f$ on $M$, since $-\frac{\partial f}{\partial r}$ is continuous on any connected components of $\partial M$, so we can call $N_i$ to be \emph{positive} (with respect to $f$) if $\left.-\frac{\partial f}{\partial r}\right|_{N_i}> 0$, and \emph{negative} if $\left.-\frac{\partial f}{\partial r}\right|_{N_i}< 0$.
\end{defi}
\par For simplicity, in this paper we only consider \emph{horn-end transverse functions} which are transverse Morse functions satisfying that on each collar neighborhood $[0,1)\times N_i$, $f|_{[0,1)\times N_i}=f|_{N_i}-k_i r$, where $k_i$ is a nonzero constant. Thus $N_i$ is positive if $k_i$ is positive and vise versa.
\begin{defi}
We define a \emph{horn-end metric} on $ M$ which has the following form on the collar neighborhood:
$$g|_{(0,1)\times\partial M}=g_{\partial M}+\frac{1}{r}dr\otimes dr.$$
Here $g_{\partial M}$ is a Riemannian metric on $\partial M$.
\end{defi}
\begin{figure}[htbp]
  \centering
  \includegraphics[width=0.5\textwidth]{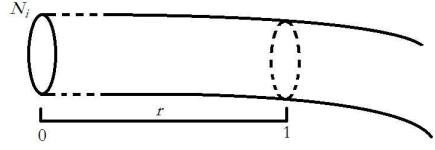}\\
  \caption{A Positive Boundary With Horn-end Metric}
  \label{Fig0}
\end{figure}
\begin{defi}
Let $X_f$ be the \emph{negative gradient vector field} of $f$, such that:
\begin{equation}
X_f=\left\{
\begin{aligned}
  &-\textrm{grad}f  ,\textrm{on } M\setminus\partial M \\
  &-\textrm{grad}\left(f|_{\partial M}\right)  , \textrm{on } \partial M
\end{aligned}
\right.
\end{equation}
And we denote by $\varphi_p(t)$ the \emph{trajectory line} generated by the vector field $X_f$ with $\varphi_p(0)=p$.
\end{defi}
\begin{lem}
$X_f$ is continuous.
\end{lem}
\begin{proof}
For a collar neighborhood $[0,1)\times N_i$,
$$df=k_i dr + d\left(f|_{N_i}\right).$$
So
$$X_f=\left\{
\begin{aligned}
  &k_ir\frac{\partial}{\partial r}-\textrm{grad}\left(f|_{N_i}\right)  ,& \textrm{on } (0,1)\times N_i \\
  &-\textrm{grad}\left(f|_{N_i}\right) & , \textrm{on } N_i
\end{aligned}
\right.,$$
thus $X_f$ is continuous.
\end{proof}
\begin{lem}
If there exist $t_0$ and $p$ such that $p\notin \partial M$ and $\varphi_p(t_0)\in \partial M$, then $t_0=\pm \infty$ and $\varphi_p(t_0)$ is a critical point of $f_{N_i}$.
\end{lem}
\begin{proof}
Assume that $t_0>0$, $\varphi_p(t_0)\in N_i$ and $\varphi_p(t)\in [0,1)\times N_i,\forall t\in [0,t_0)$. Then $\varphi_p(t)$ can be written in the form $\left(r(t),\overrightarrow{x}(t)\right)$ for $t\in [0,t_0]$. Since $X_f=k_ir\frac{\partial}{\partial r}-\textrm{grad}\left(f|_{N_i}\right)$, we can write down the equation for $\left(r(t),\overrightarrow{x}(t)\right)$:
\[
\begin{aligned}
&r'(t)=k_i r(t)\\
&\overrightarrow{x}'(t)=X_{f_{N_i}}(x(t))
\end{aligned}.
\]
From the first equation we can know that $r(t)=r(0)e^{k_i t}$. Since $p\notin N_i$, $r(0)\neq0$. So $r(t_0)=0$ requires that $t=+\infty$ and $k_i<0$. The second equation shows that $\overrightarrow{x}(t)$ is a trajectory line of $f_{N_i}$, thus $t=+\infty$ means that $\overrightarrow{x}(t)$ is a critical point. It is similar for $t_0<0$.
\end{proof}

\subsection{Stable and Unstable Manifolds}
\par Let $M$ be an $n$-dimensional oriented manifold with boundary $\partial M=\bigcup N_i$ as before, and $g$ and $f$ be horn-end Riemannian metric and a horn-end transverse Morse function on $M$, respectively. We fix an orientation of $\partial M$ in the following way. At a certain point $x\in \partial M$, let $v_{in}$ be a vector in $T_xM$ pointing 'inward'. We expand $v_{in}$ into a basis of $T_xM$ such that $(v_{in}, v_1,..., v_{n-1})$ represents the orientation of $M$ and $\{v_i\}_{i=1, 2, \ldots, n-1}$ forms a basis of $T_x\partial M$, we define the orientation of $\partial M$ as $(v_1,..., v_{n-1})$.
\par Now we introduce some notations considering the critical points. About the classification of critical points, we denote
\begin{itemize}
  \item by $Cr^\circ(f)$ the set of internal critical points, i.e. critical points of $f| _{M\setminus\partial M}$;
  \item by $Cr^+(f)$ the set of critical points on positive boundaries, i.e. critical points of $f_{N_i}$ for positive boundaries $N_i$;
  \item by $Cr^-(f)$ the set of critical points on negative boundaries, i.e. critical points of $f_{N_j}$ for negative boundaries $N_j$;
  \item by $Cr^\bullet(f):=Cr^\circ(f)\bigcup Cr^+(f)$;
  \item by $Cr(f)$ the set of all critical points.
\end{itemize}
We still use $\mu(\alpha)$ to denote the Morse index of the critical point $\alpha$ (with respect to $f$), then
\begin{itemize}
  \item for $p\in Cr^\circ(f)$, let $\mu(p)=\mu_{f|_{M\setminus\partial M}}$;
  \item for ${\gamma}\in Cr^+(f)\cap N_i$, let $\mu({\gamma})=\mu_{f|_{N_i}}+1$;
  \item for ${\delta}\in Cr^-(f)\cap N_j$, let $\mu({\delta})=\mu_{f|_{N_j}}$.
\end{itemize}
\par According to the Morse index, we are able to separate $Cr(f)$ into several mutually disjoint subsets $Cr_k(f)=\{x\in Cr(f)| \mu(x)=k\}$, in which the subscripts reflect the common Morse index of critical points contained in the subset. Similarly, we can also separate $Cr^\circ(f), Cr^+(f), Cr^-(f), Cr^\bullet(f)$ into disjoint subsets, and use the subscripts to denote the Morse indices. For instance, we use $Cr_k^\bullet(f)$ to denote the set of critical points in $Cr^\bullet(f)$ whose Morse indices are $k$.
\par Let $B^k:=\left\{\left(x_1,x_2,...,x_k\right):x_1^2+x_2^2+...+x^2_k<1\right\}$ be the $k$-dimensional open ball, and $\partial B^k:=\overline{B^k}\setminus B^k$. Moreover, we define the $k$-dimensional open half-ball $H^k:=\left\{\left(x_1,x_2,...,x_k\right):x_1^2+x_2^2+...+x^2_k<1,x_1\geq0\right\}$ and \\
$\partial H^k:=\left\{\left(x_1,x_2,...,x_k\right):x_1^2+x_2^2+...+x^2_k<1,x_1=0\right\}$.
\par Now we define stable manifolds and unstable manifolds for critical points with respect to the flow $\varphi_x(t)$. For $\alpha\in Cr(f)$, define the \emph{stable manifold} as
$$S^f_\alpha=\left\{x\in M|\lim_{t\rightarrow +\infty}\varphi_x(t)=\alpha\right\},$$
and the \emph{unstable manifold} as
$$U^f_\alpha=\left\{x\in M|\lim_{t\rightarrow -\infty}\varphi_x(t)=\alpha\right\}.$$
We omit the superscript $f$ when it does not cause confusion. We have the following proposition about stable and unstable manifolds:
\begin{prop}
(1) The stable manifold of ${q}\in Cr^\bullet(f)$ is diffeomorphic to an open ball of dimension $n-\mu({q})$;\\
(2) The unstable manifold of $p\in Cr^\circ(f)$ is diffeomorphic to an open ball of dimension $\mu(p)$;\\
(3) The unstable manifold of ${\gamma}\in Cr^+(f)$ is diffeomorphic to the half ball of dimension $\mu({\gamma})$;\\
(4) The stable manifold of ${\delta}\in Cr^-(f)$ is diffeomorphic to the half ball of dimension $n-\mu({\delta})$;\\
(5) The unstable manifold of ${\delta}\in Cr^-(f)$ is diffeomorphic to an open ball of dimension $\mu({\delta})$.
\end{prop}
\par We fix a generic Riemannian metric such that the stable manifold $S_\alpha$ and unstable manifold $U_{\alpha'}$ are transverse to each other for $\alpha, \alpha'\in Cr(f)$.
\par As for the orientations, the situations are not uniform. For all critical points, we know that the unstable and stable manifolds intersect transversely. For $p\in Cr^\circ(f)$, we will fix the orientations on $U_p$ and $S_p$ such that, if $(v_1,\ldots,v_{\mu(p)})$ gives the orientation of $T_pU_p$ and $(v_{\mu(p)+1},\ldots,v_n) $ gives the orientation of $T_pS_p$, then $(v_1,\ldots,v_n)$ gives the orientation of $T_pM$.
\par For $\gamma\in Cr^+(f)$, we first fix an orientation $(v_{in}, v_2, \ldots, v_{\mu(\gamma)})$ in $T_\gamma U_\gamma$ for $U_\gamma$, such that $v_{in}$ is a vector pointing inward of $M$, and $\{v_i\}_{i=2, \ldots, \mu(\gamma)}$ forms a basis of $T_\gamma\partial M$. Then a basis $(v_{\mu(\gamma)+1}, \ldots, v_n)$ for $T_pS_\gamma$ is the induced orientation on $S_\gamma$ if $(v_{in}, v_2, \ldots, v_{n})$ of $T_\gamma M$ represents the orientation of $M$. Thus we assign $(v_2, \ldots, v_{\mu(\gamma)})$ as a representation of the orientation for $U_\gamma\cap \partial M$.
\par  Similarly, for $\delta\in Cr^-(f)$, given an orientation $(v_1, \ldots, v_{\mu(\delta)})$ of $U_\delta$ in $T_\delta U_\delta$, we fix an orientation $(v_{in}, v_{\mu(\delta)+2}, \ldots, v_n)$ for $S_\delta$ in $T_\delta S_\delta$ such that $(v_1, \ldots, v_{in},$ $\ldots, v_n)$ represents the orientation of $M$ at $\delta$, where $v_{in}$ is a vector heading inward of $M$ and $\{v_i\}_{i=\mu(\delta)+2, \ldots, n}$ forms a basis of $T_\delta(S_\delta\cap \partial M)$. Then we induce the orientation $(v_{\mu(\delta)+2}, \ldots, v_n)$ for $S_\delta\cap \partial M$ at $\delta$. Then we will see that there is a $(-1)^{\mu(\delta)}$ difference between the orientation of $T_\delta\partial M$ and that of $(v_1, \ldots,$ $ \widehat{v}_{\mu(\delta)+1}, \ldots, v_{n})$, which is the orientation of $U_\delta$ followed by that of $S_\delta\cap \partial M$.

\subsection{Handlebody Decomposition}
\par Then we introduce the handlebody decomposition of manifolds with boundary with respect to the unstable manifolds. As preparation, the following conclusion is obvious.
\begin{lem}
Let $f$ be a horn-end transverse Morse function and $g$ be a generic Riemannian metric. If $k_1\leq k_2$, then there is no non-constant gradient trajectory from $p_1\in Cr_{k_1}(f)$ to $p_2\in Cr_{k_2}(f)$.
\end{lem}
\par Now we construct a sequence of subsets of $M$ inductively. Take $M^{-1}$ as $\partial M$. For each $p\in Cr_0(f)=Cr_0^\bullet(f)$, we pick a local coordinate  $\{x_1, x_2, \ldots, x_n\}$ centered at $p$. Denote by $B_k(p,\epsilon)$ an $n$-dimensional open disk centered at $p$, i.e. $B_n(p,\epsilon)\cong \{(x_1, x_2, \ldots, x_n)|x_1^2 +x_2^2+\ldots +x_n^2<\epsilon\}$. When $\epsilon$ is small enough, we have:
\begin{itemize}
  \item The closure of $B_n(p,\epsilon)$ are mutually disjoint;
  \item The boundary of $B_n(p,\epsilon)$ are all smooth in $M\setminus\partial M$ and transversal to the gradient trajectories.
\end{itemize}
\par Take $M^0=M^{-1}\cup \bigcup_{q\in Cr_0^\bullet(f)}B_n(q,\epsilon)$, which is the second term in our sequence of subsets of $M$. Assume that we have already constructed $M^{k-1} (1\le k\le n)$, and that each term of $M^s(0\leq s\leq k-1)$ in the sequence $\partial M = M^{-1}\subseteq M^0 \subseteq \ldots \subseteq M^{k-1}$ satisfies the following properties:
\begin{itemize}
  \item $\partial M^s$ are all smooth in $M\setminus\partial M$ and transversal to the gradient trajectories;
  \item $M^{s-1}\cup \bigcup_{q\in Cr_s^\bullet(f)}U_q$ is a deformation retract of $M_s$.
\end{itemize}
\par In this case, we construct $M^k$ as follows. For each interior critical point $p\in Cr_k^\circ(f)$, we have $U_p\setminus M^{k-1}$ is diffeomorphic to an $k$-dimensional closed ball because of the transversality conditions as well as the properties of $M^{k-1}$. There exists a tubular neighborhood $T_p$ for each $U_p\setminus M^{k-1}$, such that $T_p$ is diffeomorphic to $U_p\setminus M^{k-1} \times B_{n-k}$ and that the boundary of the union of $M^{k-1}$ and $T_p$ is smooth and transverse to the gradient vector field. For each critical point $\gamma \in Cr_k^+(f)$ on positive boundary of $M$, we have $U_{\gamma}\setminus M^{k-1}$ is diffeomorphic to $\{(x_1,x_2,\cdots ,x_{k-1},r)\mid x_1^2+x_2^2+\cdots + r^2\leq 1-\delta, r\geq 0\}$, where $\delta$ is a small enough positive real number. Then there exists a tubular neighborhood $T_{\gamma}$ for each $U_{\gamma}\setminus M^{k-1}$, such that $T_{\gamma}$ is diffeomorphic to $U_{\gamma}\setminus M^{k-1} \times B_{n-k}$ and that the boundary of the union of $M^{k-1}$ and $T_{\gamma}$ is smooth in $M\setminus\partial M$ and transverse to the gradient vector field. Moreover, we can take the tubular neighborhoods thin enough so that they are mutually disjoint. We take $M^k=M^{k-1}\cup \bigcup_{q\in Cr_k^\bullet(f)}T_q$. We point out that the above properties also hold true for $M^k$:
\begin{itemize}
  \item $\partial M^k$ are all smooth in $M\setminus\partial M$ and transversal to the gradient trajectories;
  \item $M^{k-1}\cup \bigcup_{q\in Cr_k^\bullet(f)}U_q$ is a deformation retract of $M^k$.
\end{itemize}
\par In fact, the reason why we do not include the critical points in $Cr^-(f)$ is that, for $\delta\in Cr^-(f)$, $U_{\delta}\subset \partial M$. In this case, for a thin enough tubular neighborhood $T_{\delta}$ of $U_{\delta}$, $M_{k-1}\cup T_{\delta}$ does not change the homotopy type of $M_{k-1}$. We call the sequence of subsets of $M$ the handlebody decomposition of $M$.
\par Here we describe an example to illustrate the decomposition. Consider the deformed ball $D$ of dimension two in Figure \ref{Fig1}. Let $h$ be a transverse Morse function on $D$ such that $h$ is the height function out of the collar neighborhood of the boundary and the boundary is a positive boundary. In this case, we know that there are five critical points on $D$, which are denoted by $p_1, p_2, p_3, \gamma_1, \gamma_2$ respectively. They are either internal critical points or critical points on positive boundary. Then the process of constructing the handlebody decomposition is given as follow.
\begin{figure}[htbp]
\centering
\subfigure[Step 1]{
\label{Fig1-1}
\includegraphics[width=0.3\textwidth]{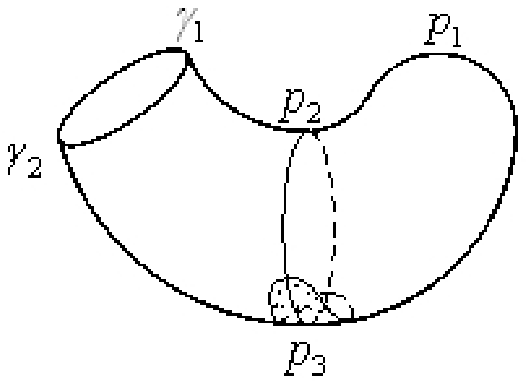}}
\subfigure[Step 2]{
\label{Fig1-2}
\includegraphics[width=0.3\textwidth]{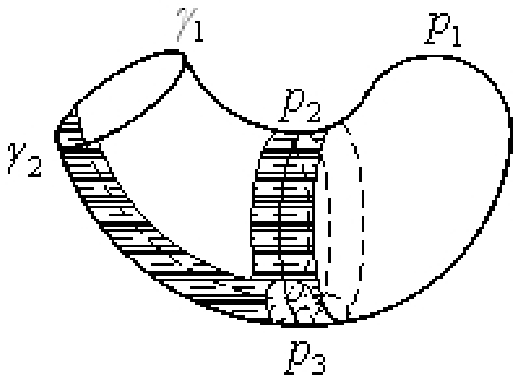}}
\subfigure[Step 3]{
\label{Fig1-3}
\includegraphics[width=0.3\textwidth]{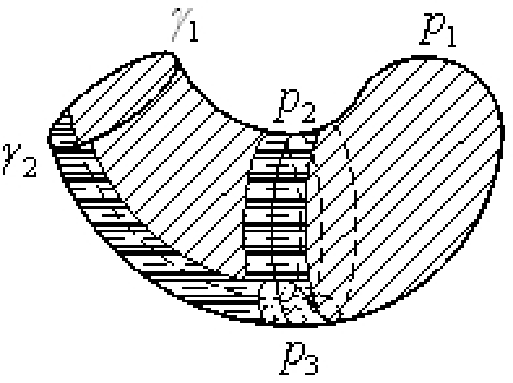}}
\caption{Handlebody Decomposition}
\label{Fig1}
\end{figure}

\section{The Relative Morse Complex}
\par In this section, we define the relative Morse complex, and then prove that its homology is isomorphic to the relative singular homology $H_*(M, \partial M)$.
\begin{defi}
We define the moduli space of all gradient trajectories between the critical points $q,q'\in Cr^\bullet(f)$ as
\[
\mathcal{M}(q,q'):=\left\{\varphi:\mathbb{R}\rightarrow M|\frac{d\varphi}{dt} = X_f\circ\varphi, \lim_{t\rightarrow-\infty}\varphi(t)=q,\lim_{t\rightarrow+\infty}\varphi(t) =q'\right\}/\sim,
\]
where the equivalence relation $\sim$ is defined as follow: two maps $\varphi_1$ and $\varphi_2$ are considered equivalent if and only if there $\exists t_0\in \mathbb{R}$ such that $\varphi_1(t_0)=\varphi_2(0)$.
\end{defi}
Having fixed generic horn-end transverse Morse functions and horn-end Riemannian metric, we have the following proposition:
\begin{prop}
$\dim\mathcal{M}(q,q')=\mu(q)-\mu(q')-1$.
\end{prop}
\begin{proof}
Since the elements in
\[
\widehat{\mathcal{M}}(q,q')=\left\{\varphi:\mathbb{R}\rightarrow M|\frac{d\varphi}{dt} = X_f\circ\varphi, \lim_{t\rightarrow-\infty}\varphi(t)=q,\lim_{t\rightarrow+\infty}\varphi(t) =q'\right\}
\]
is of 1-1 correspondence with the points in $U_q\bigcap S_{q'}$ (and it keeps the topology), and since the two submanifolds are transverse to each other,
\[
\dim\widehat{\mathcal{M}}(q,q')=\dim U_q\cap S_{q'}=\dim U_q+\dim S_{q'}-n=\mu(q)-\mu(q').
\]
As $\sim$ is a dimension-one action on $\widehat{\mathcal{M}}(q,q')$, we have
\[
\dim\mathcal{M}(q,q')=\mu(q)-\mu(q')-1.
\]
\end{proof}

\begin{defi}\label{d1}
We define a boundary operator as follow:
\[
\begin{aligned}
\partial: &\bigoplus_{q\in Cr_k^\bullet(f)}\mathbb{Z}[q]&\longrightarrow&\bigoplus_{q'\in Cr_{k-1}^\bullet(f)}\mathbb{Z}[q']\\
&\quad\quad[q]&\longmapsto&\sum_{q'\in Cr^\bullet_{k-1}(f)}\sharp\mathcal{M}(q,q')[q']
\end{aligned}.
\]
\end{defi}
\par The moduli spaces appear above are all of dimension 0, so they are all union of finite number of points, and here $\sharp$ means the number of points counted with sign, which is related to the orientation. The sign of a gradient trajectory $T$ is determined in the following way. At $x\in T$, let $e$ be the direction of $-X_f$, $(v_1, \ldots, v_{k-1}, e)$ be the orientation of $U_q$ in $T_xU_q$, and $(v_1, \ldots, v_{k-1}, e, v_{k+1}, \ldots, v_{n})$ be the orientation of $M$ such that $v_{k+1}, \ldots, v_{n}$ lie in $T_xS_{q'}$, we endow $T$ with sign `+' if $(e, v_{k+1}, \ldots, v_n)$ accords with the orientation of $S_{q'}$ in $T_xS_{q'}$, but with sign `-' if it contradicts with the orientation of $S_{q'}$. It is easy to see that this method of determining the sign is well-defined.
\begin{prop}
$\partial\partial=0$.
\end{prop}
\begin{proof}
 After a little calculation, we know that
$$\partial\partial [q]=\sum_{q''\in Cr^\bullet_{k-1}(f)}\sum_{q'\in Cr^\bullet_{k-2}(f)} \sharp\mathcal{M}(q,q'') \sharp\mathcal{M}(q'',q')[q'].$$
So we only need to prove that for any $q\in Cr^\bullet_{k}(f)$ and $q'\in Cr^\bullet_{k-2}(f)$,
$$\sum_{q''\in Cr^\bullet_{k-1}(f)} \sharp\mathcal{M}(q,q'') \sharp\mathcal{M}(q'',q')=0.$$
In fact, for a generic function $f$, the space $\mathcal{M}(q,q')$ is compactified to $\overline{\mathcal{M}}(q,q')$ such that
$$\partial\overline{\mathcal{M}}(q,q')=\bigcup_{q''\in Cr^\bullet_{k-1}(f)}\mathcal{M}(q,q'') \times\mathcal{M}(q'',q').$$
When $\mu(q)-\mu(q')=2$, we know that the dimension of $\overline{\mathcal{M}}(q,q')$ is $1$, so $\sharp \partial \overline{\mathcal{M}}(q,q')=0$.
\end{proof}
\begin{rmk}
In the definition of the chain complex, we do NOT include the critical points on the negative boundary. This is to fit the relative homology, which we will explain in the next section.
\end{rmk}
\begin{defi}
We define the relative Morse complex $CM_*(f)$ as follow:\\
(1)\quad $CM_k(f)=\bigoplus_{q\in Cr^\bullet_k(f)}\mathbb{Z}[q]$;\\
(2)\quad The boundary operator is just as in Definition \ref{d1}.\\
\end{defi}
\par Now we prove that the homology of the relative Morse complex is isomorphic to the relative singular homology. For the handlebody decomposition given in last section, we have the following proposition.
\begin{prop}
\[
H_k(M^k,M^{k-1};\mathbb{Z})\cong \bigoplus_{q\in Cr_k^\bullet(f)}\mathbb{Z}[\sigma_q],
\]
where $\sigma_q$ is the homeomorphism from a $k$-dimensional closed ball $\overline{B^k}$ to $U_q\setminus M^k$.
\end{prop}
\begin{rmk}
On one hand, the existence of $\sigma_q$ is trivial when $q$ is an interior critical point. On the other hand, the existence of $\sigma_q$ where $q$ is a boundary critical point is guaranteed by the diffeomorphism between the closure of $k$-dimensional half-ball and the unstable manifold of $U_q\setminus M^k$, and the fact that the closure of a $k$-dimensional half-ball is homeomorphic to a $k$-dimensional closed ball.
\end{rmk}
Moreover, if we denote by $\delta_k:H_k(M^k,M^{k-1}) \to H_{k-1}(M^{k-1},M^{k-2})$ the connecting homomorphism of the triple $(M^k,M^{k-1},M^{k-2})$, we get a chain complex $(H_*(M^*,M^{*-1}),\delta_*)$. The homology of the chain complex of the handlebody decomposition is isomorphic to the singular homology, which can be proved in the similar way as in CW composition. Therefore, it suffices to prove the relative Morse homology is isomorphic to the homology of the chain complex of handlebody decomposition.
\begin{prop}
The homology of $(CM_*(f),\partial_*)$ is isomorphic to the homology of $(H_k(M^k,M^{k-1}),\delta_*)$.
\end{prop}
\begin{proof}
\par We know from above that $CM_*(f)\cong H_k(M^k,M^{k-1})$ by identifying a critical point $q$ with $\sigma_q$. So all we need to prove is that the connecting homomorphisms of the two complexes are equivalent.
\par For $p\in Cr_k^\circ(f)$, the dimension of the manifold $S_{p'}\cap \sigma_p(\partial \overline{B^k})$ is zero when $p'\in Cr_{k-1}^\circ(f)$. At point $x\in S_{p'}\cap \sigma_p(\partial \overline{B^k})$, let $(u_1, \ldots, u_k)$ be a basis of $T_xU_p$ at $x$ such that $u_k$ is the direction of gradient flow of $f$ at $x$ and $(u_1, \ldots, u_{k-1})$ form a basis of $T_x\sigma_p(\partial \overline{B^k})$. When $(u_1, \ldots, u_k)$ represents the orientation of $U_p$, it induces an orientation $(u_1, \ldots, u_{k-1})$ of $\sigma_p(\partial \overline{B^k})$. Let $(v_1, \ldots, v_{n-k+1})$ be a basis of $T_xS_{p'}$ such that it represents the orientation of $S_{p'}$. Then we endow $x$ with an orientation(sign) `+' if $(u_1, \ldots, u_k, v_1, \ldots, v_{n-k+1})$ accords with the orientation of $M$, and vice versa. In this way, we are able to give each point in the dimension-zero manifold $S_{p'}\cap \sigma_p(\partial \overline{B^k})$, and thus to define the cardinality $\sharp (S_{p'}\cap \sigma_p(\partial \overline{B^k}))$. So we have, by omitting the part lying in $\partial M$,
\[
\begin{aligned}
\delta_k[\sigma_p]&=\sum_{p'\in Cr_{k-1}^\circ(f)}\sharp (S_{p'}\cap \sigma_p(\partial \overline{B^k}))[\sigma_{p'}]\\
&=\sum_{p'\in Cr_{k-1}^\circ(f)}\sharp \mathcal{M} (p,p')[\sigma_{p'}]
\end{aligned},
\]
where $\overline{B^k}$ is a $k$-dimensional closed ball.
By saying $p'\in Cr_{k-1}^\circ(f)$, we only embrace the interior critical points in the equation. However, as $\sharp \mathcal{M} (p,\gamma')$ are always zero for boundary critical points $\gamma'\in Cr_{k-1}^+(f)$, we have
\[
\delta_k[\sigma_p]=\sum_{q'\in Cr_{k-1}^\bullet(f)}\sharp \mathcal{M}(p,q')[\sigma_{q'}].
\]
\par For critical point $\gamma\in Cr_k^+(f)$, the boundary of $\sigma_{\gamma}(\partial \overline{B^k})$ can be divided into two parts: the part contained in $M^{k-2}\cup \bigcup_{p\in Cr_{k-1}^\circ(f)}T_{p}$ and the part contained in $M^{k-1}\setminus(M^{k-2}\cup \bigcup_{p\in Cr_{k-1}^\circ(f)}T_{p})$, which are denoted by $\partial_1\sigma_{\gamma}$ and $\partial_2\sigma_{\gamma}$ respectively. We have $\delta_k[\sigma_p]=[\partial_1\sigma_{\gamma}] + [\partial_2\sigma_{\gamma}]$. We determine the signs of points in $S_p\cap \sigma_{\gamma}(\partial \overline{B^k})$ in the same way as we did for interior critical points. Then
\[
[\partial_1\sigma_{\gamma}]=\sum_{p'\in Cr_{k-1}^\circ(f)}\sharp (S_p'\cap \sigma_{\gamma}(\partial \overline{B^k}))[\sigma_{p'}]
\]
\[
=\sum_{p'\in Cr_{k-1}^\circ(f)}\sharp \mathcal{M} (\gamma,p')[\sigma_{p'}].
\]
Meanwhile, we notice that points in $\partial_2\sigma_{\gamma}$ are those whose 'end point` lies in the closure of the unstable manifolds of boundary critical points, i.e. the gradient trajectories passing through these points will approach a point in the closure of the unstable manifolds of boundary critical points when time approaches the positive infinity. Moreover, due to the properties of handlebody decomposition, we have $\partial_2\sigma_{\gamma}$ is diffeomorphic to $(\sigma_{\gamma}(\partial \overline{B^k})\cap \partial M)\times [0,1)$. Let $\varphi$ be the diffeomorphism, and $t$ be the variable of the term $[0,1)$. For each $x\in (\sigma_{\gamma}(\partial \overline{B^k})\cap \partial M)\times [0,1)$, we fix the orientation $(u_1, \ldots, u_l)$ of $\partial_2\sigma_{\gamma}$ so that $((\varphi^{-1})_*(\frac{\partial}{\partial t}), u_1, \ldots, u_l)$ represents the orientation of $\sigma_{\gamma}(\partial \overline{B^k})$. It can be seen that $(\varphi^{-1})_*(\frac{\partial}{\partial t})$ is a vector heading `inward' from $\partial M$ to $M$. Moreover, we induce the orientation $(v_1, \ldots, v_{n-1})$ of $\partial M$ such that $(e, v_1, \ldots, v_{n-1})$ represents the orientation of $M$ where $e$ is the inward vector. Then we have
\[
[\partial_2\sigma_{\gamma}]=\sum_{\gamma'\in Cr_{k-1}^+(f)}\sharp (S_{\gamma'} \cap (\sigma_{\gamma}(\partial \overline{B^k})\cap \partial M))[\sigma_{\gamma'}]
\]
\[
=\sum_{\gamma'\in Cr_{k-1}^+(f)}\sharp \mathcal{M}(\gamma,\gamma')[\sigma_{\gamma'}].
\]
Therefore,
\[
\delta_k[\sigma_\gamma]=\sum_{q'\in Cr_{k-1}^\bullet(f)}\sharp \mathcal{M}(\gamma,q')[\sigma_{q'}].
\]
\end{proof}
This proposition leads to Theorem \ref{t1} mentioned in Introduction.
\par Now we still use the example of manifold $D$ and Morse function $h$ in Section 2 to illustrate the result. For the critical points, we have $p_1\in Cr_2^\circ(h), p_2\in Cr_1^\circ(h), p_3\in Cr_0^\circ(h), \gamma_1\in Cr_2^+(h), \gamma_2\in Cr_2^+(h)$. We define the orientations for the unstable manifolds as is shown on the critical points themselves in Figure \ref{Fig2}. (The orientation of the dimension-zero unstable manifold of $p_3$ is `+'.)
\begin{figure}[htbp]
\centering
\includegraphics[width=0.5\textwidth]{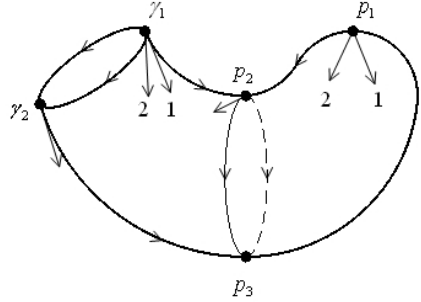}
\caption{The Critical Points}
\label{Fig2}
\end{figure}
And we define the orientation of $M$ as the same as that of $U_{p_1}$. As there are two critical points with index 2, two critical points with index 1, and one critical point with index 0, the chain groups of our Morse complex are:
\[
\cdots \rightarrow 0 \rightarrow \mathbb{Z}[p_1] \oplus \mathbb{Z}[\gamma_1] \rightarrow \mathbb{Z}[p_2] \oplus \mathbb{Z}[\gamma_2] \rightarrow \mathbb{Z}[p_3] \rightarrow 0 \rightarrow \cdots
\]
Moreover, according to the orientations, the boundary homomorphisms are
\begin{align*}
&[p_1: p_2] = 1, & [p_1: \gamma_2] = 0,\\
&[\gamma_1: p_2] = -1, & [\gamma_1: \gamma_2] = 0,\\
&[p_2: p_3] = 0, & [\gamma_2: p_3] = 1.
\end{align*}
So the homology of the relative Morse complex is
\[
H_k(CM_*(h))\cong \left\{
\begin{aligned}
\mathbb{Z}, k=1\\
0, \text{other}
\end{aligned}
\right.
\]
Therefore, we have
\[
H_k(CM_*(h))\cong H_k(D^2, \partial S^1)\cong H_k(D, \partial D).
\]

\section{Categorification}
In this part, we define an $A_\infty$-category of the relative Morse complex given above.

\subsection{Cup Product}
\par First we try to endow the relative Morse complex with product structure. Let $M$ be an $n$-dimensional oriented compact Riemannian manifold with boundary, and let $f_0$, $f_1$ and $f_2$ be three functions such that $f_0-f_1$, $f_1-f_2$ and $f_0-f_2$ are horn-end transverse Morse functions on $M$. Let $(CM_*(f_i-f_j),\partial^{i,j}_*)$ be the relative Morse complex of $f_i-f_j$. Before the description of cup product, we introduce the following definition.
\begin{defi}
For $q_1\in Cr^\bullet(f_0-f_1), q_2\in Cr^\bullet(f_1-f_2), q_3\in Cr^\bullet(f_0-f_2)$, the moduli space of $(q_1,q_2,q_3)$ is the set
\[
\mathcal{M}(q_1,q_2,q_3)=\{x\in M|\lim_{t \to -\infty}\phi^{0, 1}_t(x)=q_1, \lim_{t \to -\infty}\phi^{1, 2}_t(x)=q_2, \lim_{t \to \infty}\phi^{0, 2}_t(x)=q_3\},
\]
where $\phi^{i, j}_t$ denotes the diffeomorphism on $M$ of time $t$ generated by the gradient flow of $f_i-f_j$.
\end{defi}
In fact, as each $x\in \mathcal{M}(q_1,q_2,q_3)$ serves as a representative of the whole gradient trajectory tree with two incoming edges from $q_1,q_2$ to $x$ and one outgoing edge from $x$ to $q_3$(as is shown in Figure \ref{Fig3}), $\mathcal{M}(q_1,q_2,q_3)$ can also be treated as a moduli space of gradient trajectory trees starting from $q_1$, $q_2$ and ending at $q_3$, i.e.
\[
\mathcal{M}(q_1,q_2,q_3)\cong\left\{(x, l_1, l_2, l_3):
\begin{aligned}
& l_i:(-\infty, 0]\rightarrow M(i=1,2), l_3: [0,\infty)\rightarrow M, \\
& \lim_{t\rightarrow -\infty}l_i(t) = q_i(i=1, 2), \lim_{t\rightarrow +\infty}l_3(t) = q_3, \\
& l_1(0)= l_2(0)= l_3(0)= x\in M, \\
& l_1, l_2, l_3 \text{ are gradient trajectories of }f_0-f_1, \\
& f_1-f_2, f_0-f_2 \text{ respectively}.
\end{aligned}
\right\}.
\]
\begin{figure}[htbp]
\centering
\includegraphics[width=0.3\textwidth]{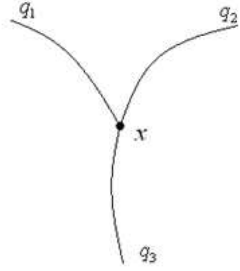}
\caption{The gradient trajectory tree corresponding to $x$}
\label{Fig3}
\end{figure}
\par In order for the moduli space to be a manifold, we need the functions to satisfy transversality conditions. To be specific, for the evaluation map
\[
E_2: U_{q_1}\times U_{q_2} \times S_{q_3} \longrightarrow M^3
\]
\[
\quad\quad\quad(x, y, z)\quad\mapsto\quad (x, y, z)
\]
we require that the image $F_2(U_{q_1}\times U_{q_2} \times S_{q_3})$ intersect transversely to $\Delta=\{(x, x, x)|x\in M\}$ in $M^3$ for all possible $q_1, q_2, q_3$. In this case, we have immediately
\[
\dim \mathcal{M}(q_1,q_2,q_3)=\mu(q_1)+\mu(q_2)-\mu(q_3) - n.
\]
As a direct corollary, we know that when $\mu(q_3)=\mu(q_1)+\mu(q_2)-n$, the dimension of $\mathcal{M}(q_1,q_2,q_3)$ is zero and thus $\mathcal{M}(q_1,q_2,q_3)$ is a union of some isolated points with orientation in $M$.
\par Now we give the definition of cup product on the chain level.
\begin{defi}
The cup product is a bilinear map
\[
m_2: CM_*(f_1-f_2)\times CM_*(f_0-f_1)\to CM_*(f_0-f_2)
\]
satisfying
\[
m_2([q_2],[q_1])=\sum_{\mu(q_3)=\mu(q_1)+\mu(q_2)-n}\sharp\mathcal{M}(q_1,q_2,q_3)[q_3].
\]
\end{defi}
In order for the cup product of the chain level induces the cup product of the homology level, we need to prove the Leibnitz rule.
\begin{thrm}
For critical point $q_1\in Cr^\bullet(f_0-f_1)$ and $q_2\in Cr^\bullet (f_1-f_2)$, we have,
\[
\partial m_2([q_2],[q_1]) + (-1)^{\mu(q_1)-n}m_2(\partial [q_2],[q_1])+m_2([q_2],\partial [q_1]) = 0.
\]
\end{thrm}
\begin{proof}
According to definition of boundary homomorphism of the Morse complex as well as cup product, we have
\[
m_2([q_2],\partial [q_1])=\sum_{\mu(q_1')=\mu(q_1)-1}\sharp \mathcal{M}(q_1,q_1')\sum_{\mu(q_3')=\mu(q_1')+\mu(q_2)-n}\sharp \mathcal{M}(q_1',q_2,q_3')[q_3'],
\]
\[
m_2(\partial [q_2],[q_1])=\sum_{\mu(q_2')=\mu(q_2)-1}\sharp \mathcal{M}(q_2,q_2')\sum_{\mu(q_3')=\mu(q_1)+\mu(q_2')-n}\sharp \mathcal{M}(q_1,q_2',q_3')[q_3'],
\]
and
\[
\partial m_2([q_2],[q_1])=\sum_{\mu(q_3)=\mu(q_1)+\mu(q_2)-n}\sharp \mathcal{M}(q_1,q_2,q_3)\sum_{\mu(q_3')=\mu(q_3)-1}\sharp \mathcal{M}(q_3,q_3')[q_3'].
\]
\par For all $q_1$,$q_2$ and $q_3'$ such that $\mu(q_3')=\mu(q_1)+\mu(q_2)-n-1$, we know $\mathcal{M}(q_1,q_2,q_3')$ is a dimension-one manifold with boundary. Moreover, when the gradient tree is approaching the boundary of the moduli space, one of the three gradient trajectories will break into two parts at a certain critical point, as is shown in Figure \ref{Fig4}.
\begin{figure}[htbp]
\centering
\subfigure[]{
\label{Fig4-1}
\includegraphics[width=0.2\textwidth]{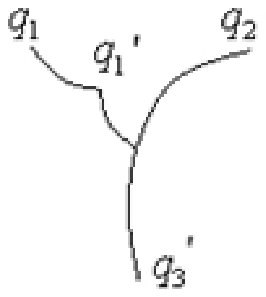}}
\subfigure[]{
\label{Fig4-2}
\includegraphics[width=0.2\textwidth]{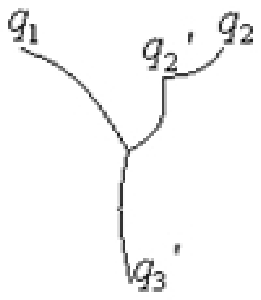}}
\subfigure[]{
\label{Fig4-3}
\includegraphics[width=0.2\textwidth]{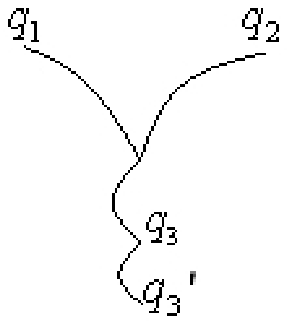}}
\caption{The Breaking of Gradient Trajectory Tree}
\label{Fig4}
\end{figure}
Notice that the breaking point must not be on negative boundaries, because there is no gradient trajectory coming out of the collar neighborhoods of negative boundaries. Therefore, we have
\[
\partial \overline{\mathcal{M}}(q_1,q_2,q_3')=\pm\sum_{\mu(q_3)=\mu(q_1)+\mu(q_2)-n} \sharp \mathcal{M}(q_1,q_2,q_3)\sharp\mathcal{M}(q_3,q_3')[q_3']
\]
\[
\quad\quad\quad\quad\quad \pm\sum_{\mu(q_1')=\mu(q_1)-1}\sharp \mathcal{M}(q_1,q_1')\sharp \mathcal{M}(q_1',q_2,q_3')[q_3']
\]
\[
\quad\quad\quad\quad\quad \pm\sum_{\mu(q_2')=\mu(q_2)-1}\sharp \mathcal{M}(q_2,q_2')\sharp \mathcal{M}(q_1,q_2',q_3')[q_3'].
\]
\par In order to determine the signs, we need to find out the difference between the orientations generated by the dimension-zero boundaries themselves and the orientations induced by $\partial \mathcal{M}(q_1,q_2,q_3')$. We first define the orientation for moduli spaces $\mathcal{M}(p_1, p_2, p_3)=U_{p_1}\cap U_{p_2}\cap S_{p_3}$. Given $p_1, p_2, p_3$, we can endow their stable manifolds and unstable manifolds with orientation as mentioned in Section 2. Then, at a point $x\in \mathcal{M}(p_1, p_2, p_3)$, assume that $(u_1, \ldots, u_r, e_1, \ldots, e_s)$ be a basis of $T_xU_{p_1}$ that reflects the orientation of $U_{p_1}$, and that $(e_1, \ldots, e_s, v_1, \ldots, v_t)$ be a basis of $T_xU_{p_2}$ that reflects the orientation of $U_{p_2}$. So, their common part $(e_1, \ldots, e_s)$ forms a basis of $T_xU_{p_1}\cap U_{p_2}$. By transversality conditions, we define the orientation of $\mathcal{M}(p_1, p_2, p_3)$ as
\[
\text{sgn}(u_1, \ldots, u_r, e_1, \ldots, e_s, v_1, \ldots, v_t)\cdot (e_1, \ldots, e_s),
\]
where $\text{sgn}(u_1, \ldots, u_{r_1}, e_1, \ldots, e_{r_2}, v_1, \ldots, v_{r_3})$ is $+1$ when $(u_1, \ldots, v_{r_3})$ is the same as the orientation of $M$ and is $-1$ when $(u_1, \ldots, v_{r_3})$ is opposite to the orientation of $M$. Then, we modify $(e_1, \ldots, e_{r_2})$ into $(e_1', \ldots, e_{r_4}', e_{1}'', \ldots, e_{r_5}'')$(where $r_2=r_4+r_5$) while not changing the orientation, such that $\{e_{1}'', \ldots, e_{r_5}''\}$ forms a basis of $T_x U_{p_1}\cap U_{p_2}\cap S_{p_3}$. We can expand them into basis $( e_1'', \ldots, e_{r_5}'', w_1, \ldots, w_{r_6} )$ of $S_{p_3}$. Then we induce the orientation of $\mathcal{M}(p_1, p_2, p_3)$ as
\[
\text{sgn}(e_1', \cdots, e_{r_4}', e_{1}'', \cdots, e_{r_5}'', w_1, \cdots, w_{r_6})\cdot (e_{1}'', \cdots, e_{r_5}'').
\]
\par Now we attempt to determine the signs in the formula of $\partial \mathcal{M}(q_1,q_2,q_3')$. First, we consider the sign of
\[
\sum_{\mu(q_1')=\mu(q_1)-1}\sharp \mathcal{M}(q_1,q_1')\sharp \mathcal{M}(q_1',q_2,q_3')[q_3'].
\]
As the sign is the difference between the orientation generated by the boundary itself and the orientation induced by $\partial \mathcal{M}(q_1,q_2,q_3')$, we consider the two orientations simultaneously. For a point $X\in \mathcal{M}(q_1,q_1') \times \mathcal{M}(q_1',q_2,q_3')$(which is a gradient tree breaking at $q_1'$), we denote by $X_1$ its part in $\mathcal{M}(q_1,q_1')$ and by $X_2$ its part in $\mathcal{M}(q_1',q_2,q_3')$. Moreover, we denote by $x_2\in M$ the corresponding point of $X_2$, i.e. $X_2$ is composed of three gradient trajectories flowing from $q_1'$ to $x_2$, from $q_2$ to $x_2$ and from $x_2$ to $q_3'$. Given the orientations of the stable and unstable manifolds of $q_1, q_2, q_3'$, we know that for the gradient trajectory tree $X$, if $X_1$ is endowed with orientation $+$, then at $x_2\in U_{q_1'}\subseteq U_{q_1}$, we have the orientations of the two manifolds are
\[
U_{q_1'}:\quad (a_1, \cdots, a_r),
\]
\[
U_{q_1}:\quad (a_1, \cdots, a_r, e),
\]
where $e$ is the `outward' vector, i.e. the vector heading to the side of $U_{q_1'}$ other than $U_{q_1}$. The reason why this happens is that, at the critical point $q_1'$, the orientation of $U_{q_1'}$ should be the orientation of $U_{q_1}$ followed by the tangent vector of the gradient trajectory flowing from $q_1$ to $q_1'$, which is actually an `outward' vector. Then we need to induce the orientations of $U_{q_1'}\cap U_{q_2}$ and $U_{q_1}\cap U_{q_2}$. In fact, we can slightly modify $e$to its projection onto $T_{x_2}U_{q_2}$ while not changing the orientations of $U_{q_1}$. The projection does not vanish because $U_{q_2}$ and $U_{q_1'}$ intersect transversely at $x_2$ but $e\notin T_{x_2}U_{q_1'}$. We still use $e$ to denote the projection. Besides, we can assume that $a_{l+1}, \ldots, a_{r}$ forms a basis of $T_{x_2}U_{q_1'}\cap U_{q_2}$ for a certain $l$. Then we can pick a specific basis of $T_{x_2}U_{q_2}$ that represents the orientation of $U_{q_2}$ and is of the form
\[
U_{q_2}: \quad (a_{l+1}, \cdots, a_{r}, e, b_1, \cdots, b_m).
\]
According to our method of inducing the orientation of the intersection of two unstable manifolds that is introduced before, we have the orientations
\[
U_{q_1'}\cap U_{q_2}: \text{sgn}(a_1, \cdots, a_r, e, b_1, \cdots, b_m)\cdot (a_{l+1}, \cdots, a_{r}),
\]
\[
U_{q_1'}\cap U_{q_2}: \text{sgn}(a_1, \cdots, a_r, e, b_1, \cdots, b_m)\cdot (a_{l+1}, \cdots, a_{r}, e).
\]
We can see that $e$ is still an outward vector. We can again modify $e$ to its projection on $S_{q_3}$ while not altering the orientation. Let $(e, c_1, \ldots, c_s)$ is a basis of $S_{q_3}$ that reflects the orientation. We know that $(a_{l+1}, \cdots, a_{r}, e, c_1, \ldots, c_s)$ forms a basis of $T_{x_2}M$ since the dimension of $U_{q_1'}\cap U_{q_2}\cap S_{q_3}$ is zero. Then the orientation of $\mathcal{M}(q_1,q_1') \times \mathcal{M}(q_1',q_2,q_3')$ at $x_2$ is
\[
\text{sgn}(a_1, \cdots, a_r, e, b_1, \cdots, b_m)\cdot \text{sgn}(a_{l+1}, \cdots, a_{r}, e, c_1, \ldots, c_s),
\]
while the orientation of $\mathcal{M}(q_1,q_2,q_3')$ is
\[
\text{sgn}(a_1, \cdots, a_r, e, b_1, \cdots, b_m)\cdot \text{sgn}(a_{l+1}, \cdots, a_{r}, e, c_1, \ldots, c_s)\cdot (e).
\]
As $e$ is the outward vector, the orientation $(e)$ induces the orientation $+$ at the boundary point, which means the orientation of $\mathcal{M}(q_1,q_1') \times \mathcal{M}(q_1',q_2,q_3')$ that is generated by itself accords with the orientation induced by $\mathcal{M} (q_1,q_2,q_3')$. In other words, for the term
\[
\sum_{\mu(q_1')=\mu(q_1)-1}\sharp \mathcal{M}(q_1,q_1')\sharp \mathcal{M}(q_1',q_2,q_3')[q_3'],
\]
the sign is $+1$.
\par Secondly, we consider the sign of the term
\[
\sum_{\mu(q_2')=\mu(q_2)-1}\sharp \mathcal{M}(q_2,q_2')\sharp \mathcal{M}(q_1,q_2',q_3')[q_3'].
\]
Similarly, for a point $Y\in \mathcal{M}(q_2,q_2') \times \mathcal{M}(q_1,q_2',q_3')$(which is a gradient tree breaking at $q_2'$), we still denote by $Y_1$ its part in $\mathcal{M}(q_2,q_2')$ and by $Y_2$ its part in $\mathcal{M}(q_1,q_2',q_3')$. Moreover, we denote by $y_2\in M$ the corresponding point of $X_2$. Then at $y_2$, we have the orientations
\[
U_{q_2'}:\quad (u_1, \cdots, u_{r'}),
\]
\[
U_{q_2}:\quad (u_1, \cdots, u_{r'}, \widetilde{e}),
\]
where $\widetilde{e}$ is still a outward vector. But the situation is different when we are inducing the orientations of $U_{q_2'}\cap U_{q_1}$ and $U_{q_2}\cap U_{q_1}$. We assume that $u_1, \ldots, u_{l'}$ forms a basis of $T_{y_2}U_{q_2'}\cap U_{q_1}$. Besides, we can modify $\widetilde{e}$ so that it lies in $T_{y_2}U_{q_1}$ while not changing the orientation of $U_{q_2}$. Now we move $\widetilde{e}$ to the first position of the basis of $T_{y_2}U_{q_2}$, i.e.
\[
U_{q_2}:\quad (-1)^{r'}(\widetilde{e}, u_1, \cdots, u_{r'}).
\]
Then we expand $\widetilde{e}, u_1, \cdots, u_{l'}$ into a basis of $T_{y_2}U_{q_1}$ such that it represents the orientation of $U_{q_1}$. To be specific, we have the orientation
\[
U_{q_1}:\quad (v_1, \cdots, v_{m'}, \widetilde{e}, u_1, \cdots, u_{l'}).
\]
Then we induce the orientations of $U_{q_1}\cap U_{q_2'}$ and $U_{q_1}\cap U_{q_2}$ as follow
\[
U_{q_1}\cap U_{q_2'}:\quad \text{sgn}(v_1, \cdots, v_{m'}, \widetilde{e}, u_1, \cdots, u_{r'})\cdot (u_1, \cdots, u_{l'}),
\]
\[
U_{q_1}\cap U_{q_2'}:\quad (-1)^{r'}\text{sgn}(v_1, \cdots, v_{m'}, \widetilde{e}, u_1, \cdots, u_{r'})\cdot (\widetilde{e}, u_1, \cdots, u_{l'}).
\]
Now we move $\widetilde{e}$ back to the end of the basis, so we have the orientation of $U_{q_1}\cap U_{q_2}$ actually is
\[
U_{q_1}\cap U_{q_2}:\quad (-1)^{r'+l'}\text{sgn}(v_1, \cdots, v_{m'}, \widetilde{e}, u_1, \cdots, u_{r'})\cdot (u_1, \cdots, u_{l'}, \widetilde{e}).
\]
Then we induce the orientation of the moduli spaces $\mathcal{M}(q_2,q_2') \times \mathcal{M}(q_1,q_2',q_3')$ and $\mathcal{M}(q_1,q_2,q_3')$, but there will be no difference to what we did concerning the gradient trajectory trees breaking at $q_1'$. So finally, for the term
\[
\sum_{\mu(q_2')=\mu(q_2)-1}\sharp \mathcal{M}(q_2,q_2')\sharp \mathcal{M}(q_1,q_2',q_3')[q_3'],
\]
the sign is $(-1)^{r'+l'}=(-1)^{r'-l'}=(-1)^{n-\mu(q_1)}=(-1)^{\mu(q_1)-n}$.
\par Lastly, in a similar but much easier way of orientation deduction, we know the sign is $+1$ for the term
\[
\sum_{\mu(q_3)=\mu(q_1)+\mu(q_2)-n}\sharp \mathcal{M}(q_1,q_2,q_3)\sharp\mathcal{M}(q_3,q_3')[q_3'].
\]
\par Therefore,
\[
\partial m_2([q_2],[q_1])+m_2(\partial [q_2],[q_1])+m_2([q_2],\partial [q_1])
\]
\[
= \sum_{\mu(q_3')=\mu(q_1)+\mu(q_2)-n}\sharp \partial\overline{\mathcal{M}}(q_1,q_2,q_3')[q_3'].
\]
As $\sharp \partial\overline{\mathcal{M}}(q_1,q_2,q_3')$ are always zero, the Leibnitz rule is proved.
\end{proof}

\subsection{Gradient Trajectory Tree}
\par The $A_\infty$-structure of Morse Category is closely related to the gradient trajectory trees of the Morse functions, so we introduce the gradient trajectory trees as preparation.
\begin{defi}
A $k$-leaf tree is a directed connected graph with $2k$ vertices, $(2k-1)$ edges, and no loops, satisfying:
\begin{itemize}
  \item There are $k$ `inward' vertices, each of which has out-degree(the number of edges originating from the vertex) $1$ and in-degree(the number of edges ending at the vertex) $0$.
  \item There is one `outward' vertex, which has out-degree $0$ and in degree $1$.
  \item When $k\neq 1$, there are $k-1$ `inside' vertices, each of which has out-degree $1$ and in-degree $2$. Otherwise, there is one `inside' vertex with out-degree $1$ and in-degree $1$.
\end{itemize}
The edges connected to the `inward' vertices or the `outward' vertices are called inward edges and outward edges respectively. Other edges are called internal edges. The `inside' vertex that is connected to the only outward edge is called the root vertex.
\end{defi}
Besides, we can endow the edges of the graph with length.
\begin{defi}
A metric $k$-leaf tree is a $k$-leaf tree in which each internal edge has a length, but each inward or outward edge is treated semi-infinite. In this case, we omit the inward and outward vertices since they are at the infinite end of the semi-infinite edges.
\end{defi}
\par It can be easily seen that there are only finitely different types of $k$-leaf trees for each $k$ up to the isomorphism of graphs, and each of these types corresponds to a component of metric $k$-leaf trees, which is homeomorphic to $(0, +\infty)^{k-2}$. For a $k$-leaf tree $T$, we denote by $\text{Com}(T)$ its corresponding component of metric $k$-leaf tree.
\par We can attach labels to the edges for the trees in the following way. When we put a $k$-leaf tree into dimension-two plane, the tree will divide the plane into $k+1$ parts. We can first label the parts in the counterclockwise order from $0$ to $k+1$. We then attach labels to the edges such that it is the combination of the labels of the two parts next to the edge. For example, the labels of the edges of a $5$-leaf tree are given in Figure \ref{Fig5}.
\begin{figure}[htbp]
\centering
\includegraphics[width=0.4\textwidth]{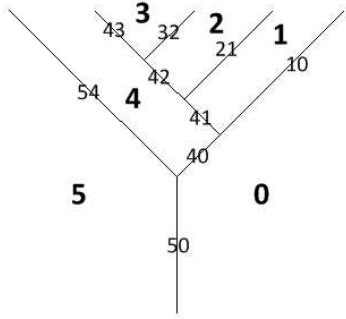}
\caption{Labels of A $5$-leaf Tree}
\label{Fig5}
\end{figure}
\par Then it comes to the definition of gradient trajectory trees.
\begin{defi}
Let $F=(f_0, f_1, \ldots, f_k)$  be a sequence of functions on $M$ such that $f_i-f_j$ is a horn-end transverse Morse function for all $i\neq j$. Then a gradient trajectory tree $F$ of shape $T$ is the image of a continuous map
\[
\tau_F: \widetilde{T} \longrightarrow M,
\]
satisfying
\begin{itemize}
  \item $\widetilde{T}$ is a metric $k$-leaf tree whose shape is $T$;
  \item the inward vertice whose inward edge is labeled $(i+1)i$ is sent to a critical point of $f_i-f_{i+1}$ and the outward vertice is sent to a critical point of $f_0-f_k$;
  \item the edge labeled $ba$ is sent to the gradient trajectory of $f_a-f_b$ with its length as the time of flowing along the gradient trajectory from one end to the other, and the direction of the edge accords with the decreasing direction of $f_a-f_b$;
  \item the images of the inward edges start at critical points, while the image of the outward edge end at a critical point(the former is called the starting points of the gradient trajectory tree, while the latter is called the end point of the tree).
\end{itemize}
\end{defi}
\subsection{$A_\infty$-Structure}
We define an $A_\infty$-category for the relative Morse Complex in this section. First we review the definition of an $A_\infty$-category.
\begin{defi}
In an $A_\infty$-category $\mathfrak{C}$, there is a class of objects $Ob(\mathfrak{C})$, and for two objects $a,b$, the morphisms from $a$ to $b$ is a chain complex $C_*(a,b)$. Moreover, $\forall a_0, a_1, \ldots, a_k\in Ob(\mathfrak{C})$, there is a linear map
\[
m_k: C_*(a_{k-1}, a_k)\otimes \cdots \otimes C_*(a_0, a_1) \rightarrow C_*(a_0, a_k)[2-k],
\]
such that
\begin{itemize}
  \item $m_1$ is the boundary homomorphism of the chain complex;
  \item the maps have the following relationships:
  \[
  (\partial m_k)(z_{k-1}\otimes\cdots\otimes z_0)=\sum_{(i, j)\in\Lambda}\pm m_{k-j+i+1}(z_{k-1}\otimes\cdots\otimes m_{j-i}(z_{j-1}\otimes \cdots\otimes z_i)\otimes\cdots\otimes z_0),
  \]
  where the index set $\Lambda=\{(i,j)|0\leq i< j\leq k, (i,j) \neq (0,k))\}$.
\end{itemize}
\end{defi}
\par Now we come back to our relative Morse homology. The Morse Category $\mathfrak{MC}(M)$ for the relative Morse complex of $M$ is defined as follow:
\begin{itemize}
  \item The set $Ob(\mathfrak{MC}(M))$ consists of the horn-end transverse Morse functions.
  \item The set of morphisms between objects $a, b$ is a modified relative Morse complex $CM_*(a-b)$, in which a critical point $q\in Cr_k^\bullet(a-b)$ is graded $|q| = k-n(n=\dim M)$.
\end{itemize}
\par Now we have already acquired a category structure. The next step is to define the linear maps so as to gain an $A_\infty$-category. For a sequence of Morse functions $F=(f_0, f_1, \ldots, f_k)$ on $M$ such that $f_i-f_j$ is a horn-end transverse Morse function for all $i\neq j$, we define the moduli space of their critical points.
\begin{defi}
Let $Q=(q_0, q_1, \ldots, q_k)$ be a sequence of points on $M$ in which $q_i\in Cr^\bullet(f_i-f_{i+1}) (i=0, 1, \ldots, k-1)$, and $q_k\in Cr^\bullet(f_0-f_k)$. The moduli space $\mathcal{M}(q_0, q_1, \ldots, q_k)$ of $Q$ is the set of gradient trajectory trees of $F$ whose starting points are $q_i(i=0, 1, \ldots, k-1)$ and whose end point is $q_k$.
\end{defi}
\par In order for the moduli spaces to be a manifold, we require that the stable and unstable manifolds of the critical points intersect transversely. Given a point $x\in U_{q_i}(i=0, 1, \ldots, k-1)$ and a $k$-leaf metric tree $T$, we are able to determine a point $y(i, x, T)\in M$ in the following way(Figure \ref{Fig6}).
\begin{figure}[htbp]
\centering
\includegraphics[width=0.4\textwidth]{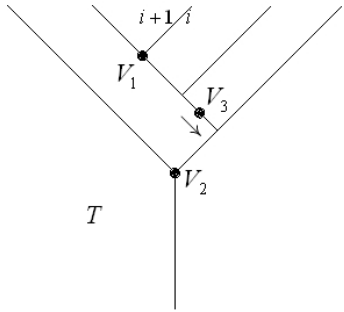}
\caption{The Movement of $V_3$}
\label{Fig6}
\end{figure}
Let $V_1$ be the inside vertex in $T$ that is connected to the incoming edge labeled $(i+1)i$, and $V_2$ be the root vertex of $T$. Now we move a point $V_3$ in the metric tree $T$ from $V_1$ to $V_2$, while at the same time we also move $x$ in $M$. When $V_3$ is moved along an edge labeled $ba$ with length $l$, we require $x$ to be moved along the gradient trajectory of $f_a-f_b$ for time $L$. Then we choose $y(x, T)$ to be the position of $x$ when $V_3$ reaches $V_2$. Based on the method of determining $y(x, T)$, we define the evaluation map
\begin{align*}
E_k: \quad\quad & U_0\times\cdots\times U_{k-1}\times S_{k}\times \text{Com}(T) && \longrightarrow\quad\quad\quad\quad M^{k+1}\\
& \quad\quad\quad(x_0, \cdots, x_k, T) && \longmapsto (y(0, x_0, T), \cdots, y(k, x_k, T)).
\end{align*}
We call the sequence of transverse Morse functions $F$ generic if the image of $E_k$ intersect transversely with $\Delta=\{(x, x, \ldots, x)|x\in M\}$ for all $Q$. For generic $F$, the moduli space $\mathcal{M}(q_0, q_1, \ldots, q_k)$ can be endowed with manifold structure. Moreover, due to the transversality condition, we have
\[
\dim (\mathcal{M}(q_0, q_1, \ldots, q_k))=\mu (q_0) + \cdots + \mu (q_{k-1}) - \mu(q_k) - (k - 1)n + k - 2,
\]
where $n$ is the dimension of $M$.
\par Then we finally come to the definition of $m_k$. For sequence of critical points $Q=(q_0, q_1, \ldots, q_k)$, $\mu(q_k) = \mu (q_0) + \cdots + \mu (q_{k-1}) - (k - 1)n + k - 2$, the dimension of $\mathcal{M}(q_0, q_1, \ldots, q_k)$ is zero. Taking $\sharp\mathcal{M}(q_0, q_1, \ldots, q_k)$ as the number of points of the moduli space with sign, we define
\[
m_k([q_{k-1}]\otimes[q_{k-2}]\otimes\cdots\otimes[q_0]) = \sum_{q_k} \sharp\mathcal{M}(q_0, q_1, \ldots, q_k)[q_k],
\]
in which the sum is over all $q_k\in Cr^\bullet(f_0-f_k)$ with the right Morse index. In this case, we have
\[
|q_k| + 2 - k = |q_0| + |q_1| + \cdots + |q_{k-1}|.
\]
\par Now we prove the $A_\infty$-relationship. We first deal with the left hand side:
\begin{align*}
 & (\partial m_k)([q_{k-1}]\otimes\cdots\otimes[q_0])\\
= & \partial(\sum_{q_k} \sharp\mathcal{M}(q_0, q_1, \ldots, q_k)[q_k])\\
= & \sum_{q_k}\sharp\mathcal{M}(q_0, q_1, \ldots, q_k) \partial([q_k])\\
= & \sum_{q_k}\sharp\mathcal{M}(q_0, q_1, \ldots, q_k) \sum_{\mu(q_{k+1}) = \mu(q_k) - 1}\sharp\mathcal{M}(q_k, q_{k+1})[q_{k+1}]\\
= & \sum_{q_{k+1}}\sum_{q_k} \sharp\mathcal{M}(q_0, q_1, \ldots, q_k) \sharp\mathcal{M}(q_k, q_{k+1}) [q_{k+1}].
\end{align*}
Besides, the compactification of each dimension 1 moduli space $\mathcal{M} (q_0, q_1, \ldots, q_{k+1})$ can be divided into the compactification of several smaller moduli spaces of gradient trajectory trees, and each of the smaller ones contains all gradient trajectory trees with a certain shape. Moreover, the boundaries of the smaller spaces either become a part of the boundary of the whole moduli space, or can also be treated as the boundaries of other smaller spaces but with opposite orientation. So the number of points (with sign) of the whole moduli space is the sum of those of the smaller ones. Let $\mathcal{M}_T(Q)$ be the moduli space of critical points sequence $Q$ that contains only the gradient trajectory trees of shape $T$, and $\Gamma$ be the set of all possible shapes of $k$-leaf gradient trajectory tree. Then, for example, we can rewrite the above equation as
\[
(\partial m_k)([q_{k-1}]\otimes\cdots\otimes[q_0]) = \sum_{q_{k+1}}\sum_{q_k}\sum_{T\in\Gamma} \sharp\mathcal{M}_T(q_0, q_1, \ldots, q_k) \sharp\mathcal{M}(q_k, q_{k+1}) [q_{k+1}].
\]
\par Then for each $T$, we consider the compactification of the dimension-one moduli space $\mathcal{M}_T(q_0, q_1, \ldots, q_{k-1}, q_{k+1})$. The gradient trajectory trees on boundary are all of the configuration that the tree breaks at one of its edge. In other words, the gradient trajectory trees on boundary becomes the union of two smaller gradient trees with the end point of one being also a starting point of the other. We call this critical point the breaking point. In fact, the breaking point cannot be on negative boundary for the same reason as we proposed in the section of cup product that once a semi-infinite gradient trajectory approaches negative boundaries, it can never come back. If we pull the matter of breaking back to the level of graph, it can be either making one of the internal edges of the $k$-leaf tree become infinite, separating the edge into two semi-infinite ones, and thus acquiring an $a$-leaf tree and a $b$-leaf tree in which $a+b=k+1$, or ,making one of the semi-infinite edges of the $k$-leaf tree become infinite, separating the edge into a two-side-infinite edge(which can also be treated as a $1$-leaf tree) and a semi-infinite edge, and thus acquiring a $1$-leaf tree and a $k$-leaf tree. We use $T=T_1\sqcup_{ij} T_2$ to express $T$ can be broken into $T_1$ and $T_2$ at the edge labeled $ij$. Then, the codimension-one boundary of the compactified moduli space $\overline{\mathcal{M}_T}(q_0, q_1, \ldots, q_{k-1}, q_{k+1})$ is
\[
\partial\overline{\mathcal{M}_T}(q_0, q_1, \ldots, q_{k-1}, q_{k+1}) =
\]
\[
\bigcup_{Y= T_1 \sqcup_{ij} T_2, q} \mathcal{M}_{T_1}(q_i, \cdots, q_{j-1}, q) \times \mathcal{M}_{T_2}(q_0, \cdots, q_{i-1}, q, q_{j}, \cdots, q_{k-1}, q_{k+1}),
\]
in which $q$ refers to internal critical points and critical points on positive boundaries such that the dimensions of the two moduli spaces are zero. Therefore,
\[
\sharp\partial\overline{\mathcal{M}_T}(q_0, q_1, \ldots, q_{k-1}, q_{k+1}) =
\]
\[
\sum_{Y= T_1 \sqcup_{ij} T_2, q} \pm \sharp \mathcal{M}_{T_1}(q_i, \cdots, q_{j-1}, q) \times \sharp \mathcal{M}_{T_2}(q_0, \cdots, q_{i-1}, q, q_{j}, \cdots, q_{k-1}, q_{k+1}).
\]
As the number of points with sign of a dimension-one manifold is always zero, we have
\[
\sum_{T\in \Gamma} \sum_{T= T_1 \sqcup_{ij} T_2, q} \pm \sharp \mathcal{M}_{T_1}(q_i, \cdots, q_{j-1}, q) \sharp \mathcal{M}_{T_2}(q_0, \cdots, q_{i-1}, q, q_{j}, \cdots, q_{k-1}, q_{k+1}) = 0
\]
holds true for all $q_{k+1}$ such that $\mu(q_{k+1}) = \mu (q_0) + \cdots + \mu (q_{k-1}) - (k - 1)n + k - 3$.
\par If we fix the breaking edge $ij$ and calculate the sum $\sharp_{ij} (q_{k+1})$ over all the possible $T, T_1, T_2, q$, we will get
\[
\sharp_{ij}(q_{k+1}) = \pm \sum_q\sharp\mathcal{M}(q_0, \cdots, q_{i-1}, q, q_j, \cdots, q_{k-1}, q_{k+1})\sharp\mathcal{M}(q_i, \cdots, q_{j-1}, q).
\]
\par It is easy to find out, according to the definitions of $A_\infty$-maps, that on one hand, when $(i,j)\neq(0,k)$, $\sharp_{ij}(q_{k+1})$ actually is the coefficient of $[q_{k+1}]$ of the term $m_{k-j+i}\circ m_{j-i+1}$, i.e.
\[
m_{k-j+i+1}([q_{k-1}]\otimes\cdots\otimes m_{j-i}([q_{j-1}]\otimes \cdots\otimes [q_i])\otimes\cdots\otimes [q_0]) = \sum_{q_{k+1}}\pm \sharp_{ij}(q_{k+1})[q_{k+1}].
\]
On the other hand, when $(i,j)=(0,k)$, we also have
\begin{align*}
& (\partial m_k)([q_{k-1}]\otimes\cdots\otimes[q_0])\\
& = \sum_{q_{k+1}}\sum_{q_k} \sharp\mathcal{M}(q_0, q_1, \ldots, q_k) \sharp\mathcal{M}(q_k, q_{k+1}) [q_{k+1}]\\
& = \sum_{q_{k+1}} \pm \sharp_{0k}(q_{k+1})[q_{k+1}].
\end{align*}
Then, because
\[
\sum_{(i,j)\in\Lambda'}\sharp_{ij}(q_{k+1}) = 0,
\]
the $A_\infty$-relationships are proved. We write the final result here.
\begin{thrm}
Let $F=(f_0, f_1, \ldots, f_k)$ be a sequence of objects in $\mathfrak{MC}(M)$, such that $f_i-f_j$ are also in $\mathfrak{MC}(M)$ for all $i\neq j$. For generic $F$, $\forall Q=(q_0, q_1, \ldots, q_{k-1})$ in which $q_s\in Cr^\bullet(f_s-f_{s+1})(s=0,1,\ldots,k-1)$, we have
\[
(\partial m_k)([q_{k-1}]\otimes [q_{k-2}] \otimes \cdots \otimes [q_0]) =
\]
\[
\sum_{(i, j)\in\Lambda}\pm m_{k-j+i+1}([q_{k-1}] \otimes \cdots \otimes m_{j-i}([q_{j-1}] \otimes \cdots \otimes [q_i]) \otimes \cdots \otimes [q_0]),
\]
where the index set is $\Lambda=\{(i,j)|0\leq i < j\leq k, (i,j) \neq (0,k))\}$, and the maps are
\[
m_k([q_{k-1}]\otimes[q_{k-2}]\otimes\ldots\otimes[q_0]) = \sum_{q_k} \sharp\mathcal{M}(q_0, q_1, \ldots, q_k)[q_k].
\]
\end{thrm}
\par This theorem is equivalent to Theorem \ref{t2}.

\section*{Acknowledgements}
The authors would like to thank the academic advisor Huijun Fan for the guidance on mathematics at every stage of the research and the kind encouragement all the time. The authors were supported by the Huabao Foundation of Peking University.


\end{document}